\documentclass[a4paper, 12pt]{amsart}
\usepackage{amsfonts, amsthm, amssymb, amsmath, stmaryrd}
\usepackage{mathrsfs,array}
\usepackage{xy}
\usepackage{verbatim}
\usepackage{amscd} 
\usepackage{tikz}
\usepackage{tikz-cd}
\usetikzlibrary{matrix,arrows,decorations.pathmorphing}
\usepackage{textcomp}
\usepackage{mathtools}
\input xy
\xyoption{all}
\usepackage[unicode]{hyperref}

\usepackage{etoolbox}
\makeatletter
\patchcmd{\@bibitem}{\ignorespaces}{\label{bib-#1}\ignorespaces}{}{}
\makeatother

\numberwithin{equation}{subsection}

\newtheorem{thm}{Theorem}[subsection]

\newtheorem{cor}[thm]{Corollary}

\newtheorem{lem}[thm]{Lemma} 
\newtheorem{prop}[thm]{Proposition}
 \theoremstyle{definition}
 \theoremstyle{definition}
\newtheorem{defn}[thm]{Definition} \theoremstyle{remark}
\newtheorem{rem}[thm]{\bf Remark}

\newtheorem{para}[thm]{\bf}

\DeclareMathOperator{\Hom}{Hom}
\DeclareMathOperator{\End}{End}
\DeclareMathOperator{\Ext}{Ext}

\DeclareMathOperator{\Spa}{Spa}

\DeclareMathOperator{\Fil}{Fil}
\DeclareMathOperator{\gr}{gr}

\DeclareMathOperator{\Lie}{Lie}

\DeclareMathOperator{\Res}{Res}

\DeclareMathOperator{\Sym}{Sym}

\def\inf{\mathrm{inf}}

\def\Ind{\mathrm{Ind}}

\def\an{\mathrm{an}}
\def\la{\mathrm{la}}

\def\Sh{\mathrm{Sh}}

\newcommand{\Z}{\mathbb{Z}}
\newcommand{\F}{\mathbb{F}}

\newcommand{\Q}{\mathbb{Q}}
\newcommand{\R}{\mathbb{R}}

\newcommand{\A}{\mathbb{A}}
\newcommand{\bC}{\mathbb{C}}

\newcommand{\rmG}{\mathrm{G}}

\newcommand{\GL}{\mathrm{GL}}
\newcommand{\SL}{\mathrm{SL}}
\newcommand{\PGL}{\mathrm{PGL}}

\newcommand{\Fl}{{\mathscr{F}\!\ell}}

\newcommand{\cO}{\mathcal{O}}
\newcommand{\cL}{\mathcal{L}}

\newcommand{\RNum}[1]{\uppercase\expandafter{\romannumeral #1\relax}}

\begin{document}

\title{Completed cohomology of Hilbert modular varieties below middle degree}
\author{Lue Pan}

\begin{abstract} 
We study the locally analytic vectors of completed cohomology of Hilbert modular varieties below middle degree. Our main result says that the universal enveloping algebra action has to factor through certain quotients. Application includes bounds on the Gelfand-Kirillov dimension and a density result. The method is a continuation of our joint work with Kai-Wen Lan.
\end{abstract}

\maketitle

\tableofcontents

\section{Introduction}
\begin{para}
Throughout this paper we fix a prime number $p$ and a totally real number field $F$ of degree $d$ over $\Q$. Let $\rmG=\mathrm{Res}_{F/\Q}\GL_2$ be the restriction of scalars of $\GL_2$ from $F$ to $\Q$. In particular $\rmG(\R)=\prod_{\tau:F\to \R} \GL_2(\R)$, where $\tau$ runs through all embeddings of $F$ into $\R$, and it acts naturally on $\mathrm{X}:=\prod_{\tau:F\to \R} (\bC\setminus\R)$ by the usual M\"obius transformation. 
For a neat compact open subgroup $K$ of $\rmG(\A_f)$ with $\A_f$ denoting the ring of finite ad\`eles, we have the usual (adelic) Hilbert modular variety
\[\Sh_{K,\bC}:=\rmG(\Q)\setminus (\mathrm{X}\times \rmG(\A_f))/K.\]
When $K$ varies,  $\{\Sh_{K,\bC}\}_{K\subseteq \rmG(\A_f)}$ forms a tower with a natural action of $\rmG(\A_f)$. It is well-known (and easy to prove by approximation theorem \cite[\S 2]{De71a}) that the action of $\rmG(\A_f)$ on the set of connected components $\{\pi_0(\Sh_{K,\bC})\}_{K\subseteq \rmG(\A_f)}$ factors through the morphism $\mathrm{Res}_{F/\Q} \det: \mathrm{Res}_{F/\Q}\GL_2 \to \mathrm{Res}_{F/\Q}\GL_1$, or equivalently, the action of $\SL_2(\A_f\otimes_{\Q} F)\subseteq \rmG(\A_f)$ on the $0$-th Betti cohomology group $\varinjlim_K H^0(\Sh_{K,\bC}, \Q_p)$ is trivial. The purpose of this note is to generalize this result to cohomology of higher degree, and discuss some consequences.
\end{para}

\begin{para}
The cohomology group we are going to consider is the completed cohomology group introduced by Emerton in \cite{Eme06}. Let $K^p$ be an open compact subgroup of $\rmG(\A_f^p)$, where $\A_f^p$ denotes the finite ad\`eles away from $p$. The completed cohomology with tame level $K^p$ is defined as
\[
\tilde{H}^i(K^p):=\varprojlim_n\varinjlim_{K_p\subseteq \rmG(\Q_p)} H^i(\Sh_{K^pK_p,\bC},\Z/p^n)
\]
where $K_p$ runs through all open compact subgroups of $\rmG(\Q_p)$. We will only consider the rational completed cohomology $\tilde{H}^i(K^p)_{\Q_p}:=\tilde{H}^i(K^p)\otimes_{\Z_p} \Q_p$ below. It is a unitary admissible $p$-adic Banach space representation of $\rmG(\Q_p)$ with unit ball given by the torsion-free quotient of $\tilde{H}^i(K^p)$. Denote by
\[\tilde{H}^i(K^p)_{\Q_p}^{\la}\subseteq \tilde{H}^i(K^p)_{\Q_p}\]
the subspace of $\rmG(\Q_p)$-locally analytic vectors. The Lie algebra $\Lie \rmG(\Q_p)$ of the $p$-adic Lie group $\rmG(\Q_p)$ acts on it via derivation. Fix $E$ a finite extension of $\Q_p$ which contains all embeddings of $F$ into an algebraic closure of $E$. Let
\[\tilde{H}^i(K^p)_{E}^{\la}=\tilde{H}^i(K^p)_{\Q_p}^{\la}\otimes_{\Q_p} E.\]
It is an $E$-linear representation of the Lie algebra 
\[\Lie \rmG_{E}:=\Lie \rmG\otimes_{\Q} E = \prod_{ \Sigma_p} \mathfrak{gl}_2(E)\]
where $ \Sigma_p$ is the set of embeddings $\tau:F\to E$. Then $|\Sigma_p|=d$ by our assumption on $E$. We will focus on the action of $\prod_{ \Sigma_p} \mathfrak{sl}_2(E)$ inside and denote by $\mathfrak{sl}_{2,\tau}(E)$ its $\tau$-component for $\tau\in\Sigma_p$.
Here is the main result.
\end{para}

\begin{thm} \label{MT}
For $i\in\{0,\cdots,d-1\}$, there is a finite $\Lie \rmG_E$-invariant filtration $\Fil^\bullet$ on $\tilde{H}^i(K^p)_{E}^{\la}$ with the following property
\begin{itemize}
\item for each $j$, the set of $\tau\in\Sigma_p$ such that $U(\mathfrak{sl}_{2,\tau}(E))$ acts via a finite-dimensional $E$-algebra on $\gr^j \tilde{H}^i(K^p)_{E}^{\la}$ has cardinality at least $d-i$.
\end{itemize}
\end{thm}

This has immediate consequence on the Gelfand-Kirillov dimension of the torsion-free quotient of $\tilde{H}^i(K^p)$. To ease notation in the introduction, we only consider the case of trivial central character here. See Theorem \ref{mtj} and Corollary \ref{ccden} for more general statements. 

We write $F_{p}=F\otimes_{\Q}\Q_p$ and $\cO_{F,p}=\cO_F\otimes_{\Z}\Z_p$. Assume that $K^p\GL_2(\cO_{F,p})$ is neat.

\begin{cor}
Let $\tilde{H}^i(K^p)_{1}\subseteq \tilde{H}^i(K^p)$ be the subspace fixed by $\cO_{F,p}^\times$ in the center of $\rmG(\Q_p)$. It is an admissible $p$-adic representation of $\PGL_2(\cO_{F,p})$.
\begin{enumerate}
\item $\Hom_{\Z_p}(\tilde{H}^i(K^p)_{1},\Q_p)$ is a finitely-generated $(\Z_p[[\PGL_2(\cO_{F,p})]]\otimes_{\Z_p}\Q_p)$-module of codimension at least $3(d-i)$, cf. \cite[\S1.2]{CE12} for the notion of codimension.
\item The subspace of $\PGL_2(\cO_{F,p})$-algebraic vectors in $\tilde{H}^d(K^p)_{1}\otimes\Q$ is dense.
\end{enumerate}
\end{cor}

\begin{para}
The proof of the main result will be given in the next section. The very rough idea is as follows. Assume for simplicity $F=\Q$, i.e. $d=1$. In our previous work \cite{Pan22}, we show that the (complexified) locally analytic completed cohomology can be localized on $\mathbb{P}^1$ and it is naturally isomorphic to $H^*(\Fl,\cO^{\la})$, where $\Fl$ denotes the adic flag variety of $\GL_2$ hence isomorphic to $\mathbb{P}^1$, and $\cO^{\la}$ is some $\tilde{\mathscr{D}}$-module on it. In a joint work with Kai-Wen Lan \cite{LP25}, we deduce from Bhatt's mixed characteristic Kodaira vanishing result that $H^{<1}(\Fl,\cO^{\la}\otimes_{\cO_{\Fl}} \mathcal{L}^{-1})=0$ for the minimal ample line bundle $\mathcal{L}$ on $\Fl$. 
Now from the point of view of \cite[\S 2]{BG99}, the twist of $\mathcal{L}^{-1}$ intertwines with certain translation functor when computing the cohomology. More precisely, starting with a $\tilde{\mathscr{D}}$-module, we can first twist it by a line bundle and take its cohomology (e.g. $H^{<1}(\Fl,\cO^{\la}\otimes_{\cO_{\Fl}} \mathcal{L}^{-1})$). This will be the same as first taking cohomology then applying a  translation functor determined by $\mathcal{L}$ \cite[Prop.2.8]{BG99}.
 We conclude that $H^{<1}(\Fl,\cO^{\la})$ is in the kernel of some translation functor, which turns out to be exactly those killed by $\mathfrak{sl}_2$ in this case. 

It should be clear that this argument will work for general Shimura variety as well. 
\end{para}

Notation: for $G$ a topological group $G$ and $A$ a topological vector space, $C(G,A)$ denotes the space of continuous $A$-valued functions on $G$ viewed as the regular $G$-representation by right translation. 

\subsection*{Acknowledgement} 
The author is partially supported by a Sloan Research Fellowship. He would like to thank Kai-Wen Lan and Yuanyang Jiang for useful conversations and thank the anonymous referee for helpful comments and pointing out a mistake in the first version.

\section{Proof of main result}
\subsection{Recollection of results from \cite{LP25}}
\begin{para}
We review some results obtained in \cite{LP25}, which are based on Bhatt's mixed characteristic Kodaira vanishing result and Rodriguez Camargo's work on the geometric Sen theory over general Shimura variety \cite{RC24}. First we introduce some notations on the Lie theory side.

Recall that $\mathrm{X}=\prod_{\tau:F\to \R} (\bC\setminus\R)$. There is a $\rmG(\R)$-equivalent embedding
\[\prod_{\tau:F\to \R} (\bC\setminus\R)\subseteq \prod_{\tau:F\to \R}\mathbb{P}^1(\bC)\] 
(usually called the Borel embedding) of $\mathrm{X}$ into the $\bC$-points of the flag variety of $\rmG$. The stabilizer of $\rmG(\bC)$-action on $\prod_{\tau:F\to \R}\mathbb{P}^1(\bC)$ at the point  $(\sqrt{-1},\cdots,\sqrt{-1})$ \footnote{We use $\sqrt{-1}$ to denote the usual $i\in \bC$ here because the letter $i$ will be reserved for denoting cohomology degrees later.} defines a parabolic subgroup $\mathrm{P}^{\mathrm{std}}_{\bC}$ of $\rmG_{\bC}:=\rmG\times_{\Q}\bC$. The base change from $\R$ to $\bC$ of its real points defines a Levi subgroup denoted by $\mathrm{M}_{\bC}$. The opposite parabolic of $\mathrm{P}^{\mathrm{std}}_{\bC}$ with respect to $\mathrm{M}_{\bC}$ will be denoted by $\mathrm{P}_{\bC}$.

Let $C$ denote the completion of an algebraic closure $\overline{\Q}_p$ of $\Q_p$ and fix an isomorphism $C\cong \bC$. Under this isomorphism, $\mathrm{M}_\bC$ determines a subgroup $\mathrm{M}$ of $\rmG_C$. We denote its Lie algebra by $\mathfrak{h}$,\footnote{We use the notation $\mathfrak{h}$ here instead of $\mathfrak{m}$ in the reference because $\mathfrak{m}$ is a Cartan subalgebra in this case.} and define $\mathrm{P}$ and $\mathfrak{p}$ similarly associated to $\mathrm{P}_\bC$. Let $\Fl$ be the adic space over $\Spa(C,\cO_C)$ associated to the flag variety $\rmG_C/\mathrm{P}$. Concretely $\Fl$ is isomorphic to product of $d$ copies of (adic) projective lines. Given a finite-dimensional representation $W$ of $\mathrm{P}$, as usual we get a $\rmG(C)$-equivariant vector bundle  on $\Fl$ which will be denoted by $\mathcal{W}$. In particular, an integral weight $\lambda\in \mathfrak{h}^*$ defines a character of $\mathrm{M}$ hence a character of $\mathrm{P}$ via the natural quotient morphism $\mathrm{P}\to\mathrm{M}$, and thus we get an equivariant line bundle $\mathcal{L}_\lambda$ on $\Fl$. The Lie algebra $\mathfrak{g}:=\Lie\rmG(C)$ acts naturally on all these equivariant bundles (particularly $\cO_\Fl$) via derivation. Denote by $\mathfrak{n}$ the nilradical of $\mathfrak{p}$. Consider the following sequence of $\mathrm{P}$-representations
\[\mathfrak{n}\subseteq \mathfrak{p}\subseteq \mathfrak{g}\]
and the associated equivariant vector bundles 
\[\mathfrak{n}^0\subseteq  \mathfrak{p}^0\subseteq \mathfrak{g}^0.\]
Since the action of $\mathrm{P}$ on $\mathfrak{g}$ extends naturally to $\rmG_C$, there is a natural isomorphism
\[\mathfrak{g}^0\cong  \cO_\Fl \otimes_C \mathfrak{g}.\]
On the other hand, $\mathrm{P}$ acts trivially on $\mathfrak{p}/\mathfrak{n}\cong \mathfrak{h}$. Thus $ \mathfrak{p}^0/ \mathfrak{n}^0$ is a trivial vector bundle and 
\[H^0(\Fl, \mathfrak{p}^0/ \mathfrak{n}^0)=\mathfrak{h}.\]
\end{para}

\begin{para}
Let $\Sigma$ denote the set of embeddings $F\to C$. Then there is a natural decomposition
\[\mathfrak{g}=\prod_{\tau\in\Sigma}\mathfrak{g}_\tau\] 
with $\mathfrak{g}_\tau= \mathfrak{gl}_2(C)$.
Similarly we have $\mathfrak{n}= \prod_{\tau\in\Sigma}\mathfrak{n}_\tau$, $\mathfrak{h}= \prod_{\tau\in\Sigma}\mathfrak{h}_\tau$ and $\Fl=\prod_{\tau\in\Sigma}\Fl_{\tau}$. Let $\mathfrak{sl}_{2,\tau}:=\mathfrak{sl}_2(C)\subseteq \mathfrak{g}_\tau$.
\end{para}

\begin{defn} \label{defnpw}
\hspace{2em}
\begin{enumerate}
\item For $\tau\in\Sigma$, we define $H_\tau\in\mathfrak{sl}_{2,\tau}\cap\mathfrak{h}_\tau$ as follows. Consider the standard representation $\mathrm{std}$ of $\mathfrak{gl}_2(C)$. As a $\mathfrak{p}_\tau$-representation, it is a non-split extension of one-dimensional representations
\[0\to (1,0)\to \mathrm{std}\to (0,1) \to 0.\]
We require that $H_\tau$ acts on $(1,0)$ via $-1$ and on $(0,1)$ via $1$.
\item The center $Z(\mathfrak{g}_\tau)$ is naturally isomorphic to $C$. We will write an integral weight $\lambda\in\mathfrak{h}^*$ as
 \[\lambda=(\{k_\tau\}_{\tau\in\Sigma};w)\]
if $k_\tau=\lambda(H_\tau)$, and
 $\lambda|_{Z(\mathfrak{g}_\tau)}: Z(\mathfrak{g}_\tau)=C\to C$ is multiplication by $w\in \Z$  for all $\tau\in\Sigma$. (Only weights of this form will be considered below.) It's easy to see that $k_\tau$ and $w$ must have the same parity. We say $(\{k_\tau\}_{\tau\in\Sigma};w)$ is \textit{parallel} if all $k_\tau$'s are the same and write it as $(k;w)$ with $k=k_\tau$.
 \item If $\lambda=(\{k_\tau\}_{\tau\in\Sigma};w)$ satisfies $k_\tau\geq 0$ for all $\tau$, we denote by $V_\lambda$ the $\mathfrak{g}$-representation of \textit{lowest weight} $\lambda$ with respect to $\mathfrak{p}$. Hence
 \[V_\lambda\cong \bigotimes_\tau (\Sym^{k_\tau}\mathrm{std} \otimes \det{}^{(w-k_\tau)/2}).\]
\end{enumerate} 
\end{defn}
 
\begin{para} \label{ShL}
Given an integral weight $\lambda\in\mathfrak{h}^*$, it defines an equivariant line bundle on $\prod_{\tau:F\to \R}\mathbb{P}^1(\bC)$ by viewing $\prod_{\tau:F\to \R}\mathbb{P}^1_\bC$ as the flag variety of $\rmG_\bC$ associated to the parabolic $\mathrm{P}^{\mathrm{std}}_{\bC}$. We can restrict this equivariant line bundle to $\mathrm{X}$ via the Borel embedding, and it descends to a line bundle $L_{\Sh,\lambda}$ on Hilbert modular variety $\Sh_{K,\bC}$ for a neat level $K$ if $\lambda$ is parallel.
\end{para}

\begin{para}
Finally we relate completed cohomology to $\Fl$. The following is a consequence of \cite[Theorem 6.2.6]{RC24} and \cite[Theorem 4.3.4]{LP25}.
\end{para}

\begin{thm} \label{thmLP}
There exists an object $\cO^{\la}$ in the bounded derived category of sheaves of $\cO_\Fl$-modules on $\Fl$ equipped with a Lie algebra action of $\mathfrak{g}$ when viewed as an object in the bounded derived category of sheaves of abelian groups. The $\mathfrak{g}$-module structure is compatible with the $\cO_\Fl$-module structure, i.e. $[l,f]=l(f)$ holds when acting on $\cO^{\la}$ for $l\in \mathfrak{g}$ and $f\in\cO_\Fl$. Moreover $\cO^{\la}$ satisfies following properties.
\begin{enumerate}
\item Differential equations: Consider the action of $\mathfrak{g}^0\cong  \cO_\Fl \otimes_C \mathfrak{g}$ on $\cO^{\la}$ by $\cO_\Fl$-linearly extending the action of $\mathfrak{g}$. This action is zero when restricted to the subsheaf $\mathfrak{n}^0$.
\item Comparison with locally analytic completed cohomology: There is a $\mathfrak{g}$-equivariant isomorphism $H^i(\Fl,\cO^{\la})\cong \tilde{H}^i(K^p)^\la\widehat\otimes_{\Q_p} C$, where the right hand side is computed as the direct limit of $p$-adically completed tensor product of the subspace of $K_p$-analytic vectors $\tilde{H}^i(K^p)^{K_p-\an}\subseteq \tilde{H}^i(K^p)$ and $C$ with $K_p$ running through open compact subgroups of $\rmG(\Q_p)$.
\item Kodaira vanishing: Suppose $\lambda=(k;w)$ is a parallel weight. Then
\[H^{<d}(\Fl, \cO^{\la}\otimes^\mathbb{L}_{\cO_\Fl}\mathcal{L}_{\lambda}^{-1})=0\]
if $k>0$. (Following the reference, we will drop the ${}^\mathbb{L}$ in the derived tensor product from now on.)
\end{enumerate}
\end{thm}

\begin{proof}
Only the last part needs extra explanation. To apply \cite[Theorem 4.3.4]{LP25}, we need to check that  a positive power of the line bundle $L_{\Sh,\lambda}$ (introduced in \ref{ShL}) extends to an ample line bundle on its minimal compactification. This is well-known (for example by Baily-Borel) as the weight is parallel.
\end{proof}

\begin{rem}
In the case of modular curves, $\cO^{\la}$ was denoted as $\cO^{\la}_{K^p}$ in \cite{Pan22}.
\end{rem}

\begin{para}
We will make use of the horizontal action in the proof. See \cite[\S 4.4, 5.2]{LP25} for more details. By the second part of Theorem \ref{thmLP}, the action of $\mathfrak{p}^0$ on $\cO^{\la}$ factors through the quotient $\mathfrak{p}^0/\mathfrak{n}^0$. Hence we get actions of $H_\tau\in \mathfrak{h}=H^0(\Fl,\mathfrak{p}^0/\mathfrak{n}^0)$ on $\cO^{\la}$ for $\tau\in\Sigma$. These operators are $\cO_{\Fl}$-linear because $\mathfrak{p}^0$ acts trivially on $\cO_{\Fl}$. In particular, for a representation $W$ of $\mathfrak{p}$, we can also define actions of $H_\tau$ on $\cO^{\la}\otimes_{\cO_{\Fl}} \mathcal{W}$ by acting solely on $\cO^{\la}$. Clearly this action is functorial in $W$. This will be called the horizontal action and should be distinguished from the constant action of $\mathfrak{h}\subseteq\mathfrak{g}$ on $\cO^{\la}$. It commutes with the action of $\mathfrak{g}$.

On the other hand, the diagonal action of $\mathfrak{g}$ on $\cO^{\la}\otimes_{\cO_{\Fl}} \mathcal{W}$ extends to the universal enveloping algebra $U(\mathfrak{g})$. We denote by $\Omega_{\tau}\in Z(U(\mathfrak{sl}_2(C)))\subseteq U(\mathfrak{g})$ the Casimir operator of the $\tau$-th component. Its relationship with the action of $H_\tau$ is as follows.
\end{para}

\begin{thm} \label{HCas}
Suppose $\lambda$ is an integral weight of $\mathfrak{h}$. The operators $H_\tau$ and $\Omega_\tau$ satisfy
\[\Omega_\tau=\frac{1}{2}(H_\tau+1+\lambda(H_\tau))^2-\frac{1}{2}\]
when acting on $\cO^{\la}\otimes_{\cO_{\Fl}} \mathcal{L}_\lambda$.
\end{thm}

Below we will view $\cO^{\la}\otimes_{\cO_{\Fl}} \mathcal{L}_\lambda$ as a $(\mathfrak{h}\times\mathfrak{g})$-module with respect to the horizontal action of $\mathfrak{h}$ and the (constant) action of $\mathfrak{g}$.

\begin{proof}
This follows from \cite[Theorem 4.4.4]{LP25} and the discussion in the beginning of \S 5.2.
\end{proof}

\subsection{The case of modular curve}
\begin{para}
We explain a proof of Theorem \ref{MT} in the case of modular curve, i.e. $F=\Q$ in this subsection. Even though a much simpler and direct argument exists in this case, I hope that this subsection gives some basic idea of the argument in the next subsection. Since $d=1$ and $\Sigma$ has only one element now, we will  drop the $\tau$'s in $H_\tau$ and $\Omega_\tau$. 
\end{para}

\begin{prop} \label{stpM1}
$H^0(\Fl,\cO^{\la})$ is annihilated by the horizontal operator $H$.
\end{prop}

\begin{proof}
This is proved in \cite{LP25}. We recall the argument here. Consider the sequence of $\mathfrak{p}$-representations
\[0\to (1,0)\to \mathrm{std}\to (0,1) \to 0\]
and the associated sequence of equivariant vector bundles on $\Fl$
\[0\to \mathcal{L}_{(1,0)}\to \cO_{\Fl}\otimes_C \mathrm{std}\to \mathcal{L}_{(0,1)} \to 0.\]
Tensor this sequence with $\mathcal{L}_{(0,1)}^{-1}\otimes_{\cO_\Fl} \cO^{\la}$ and take cohomology. We get an exact sequence
\[ H^0(\Fl,\cO^{\la}\otimes_{\cO_\Fl} \mathcal{L}_{(0,1)}^{-1})\otimes_C\mathrm{std} \to H^0(\Fl,\cO^{\la}) \to H^1(\Fl,\cO^{\la}\otimes_{\cO_\Fl} \mathcal{L}_{(1,0)}\otimes_{\cO_\Fl} \mathcal{L}_{(0,1)}^{-1}).
\]
By the Kodaira vanishing in Theorem \ref{thmLP}, the first term is zero. Hence the second map is injective. Now we observe that this map is also $(H,\Omega)$-equivariant, and by Theorem \ref{HCas}, 
\begin{itemize}
\item $\Omega=\frac{1}{2}(H+1)^2-\frac{1}{2}$ on $\cO^{\la}$,
\item $\Omega=\frac{1}{2}(H+1-2)^2-\frac{1}{2}$ on $\cO^{\la}\otimes_{\cO_\Fl}\mathcal{L}_{(1,0)}\otimes_{\cO_\Fl} \mathcal{L}_{(0,1)}^{-1}$.
\end{itemize}
We deduce that $\frac{1}{2}(H+1)^2-\frac{1}{2}=\frac{1}{2}(H+1-2)^2-\frac{1}{2}$ on $H^0(\Fl,\cO^{\la})$. It follows that $H=0$.
\end{proof}

\begin{para} \label{bard}
Denote by $\cO^{\la,H=0}$ the fiber of the morphism $\cO^{\la} \stackrel{H}{\to} \cO^{\la}$ in $D^b(\cO_{\Fl})$ (note that $H$ is $\cO_\Fl$-linear). Although both $\cO^{\la}$ and $\cO^{\la,H=0}$  are sheaves in the case of modular curve, cf. \cite[Lemma 5.1.2]{Pan22}, we still use the language of derived category to be aligned with the general Hilbert case. 
Let $\mathfrak{g}^0_0=\cO_\Fl\otimes_{C} \mathfrak{sl}_2(C)\subseteq \mathfrak{g}^0$ and $\mathfrak{p}^0_0=\mathfrak{p}^0\cap \mathfrak{g}^0_0$. Then $H$ generates $\mathfrak{p}^0_0/\mathfrak{n}^0$. Therefore $\mathfrak{p}^0_0$ acts trivially on $\cO^{\la,H=0}$ and the action of $\mathfrak{g}^0_0$ on $\cO^{\la,H=0}$ induces a natural $\mathfrak{g}$-equivariant morphism
\[\bar{d}: \cO^{\la,H=0}\to \cO^{\la,H=0}\otimes_{\cO_\Fl} (\mathfrak{g}^0_0/\mathfrak{p}^0_0)^{-1}. \]
(This is only a morphism in the bounded derived category of sheaves of abelian groups on $\Fl$ as it is not $\cO_\Fl$-linear.)
The line bundle $\mathfrak{g}^0_0/\mathfrak{p}^0_0$ is isomorphic to the tangent bundle of $\Fl$, hence ample. Its dual $(\mathfrak{g}^0_0/\mathfrak{p}^0_0)^{-1}$ is isomorphic to the cotangent bundle $\Omega_\Fl\cong \mathcal{L}_{(-1,1)}^{-1}$.
\end{para}

\begin{prop} \label{stpM21}
$H^0(\Fl,\cO^{\la,H=0}\otimes_{\cO_\Fl} (\mathfrak{g}^0_0/\mathfrak{p}^0_0)^{-1})=0$.
\end{prop}

\begin{proof}
Consider the exact triangle \[\cO^{\la,H=0}\otimes_{\cO_\Fl} (\mathfrak{g}^0_0/\mathfrak{p}^0_0)^{-1} \to \cO^{\la}\otimes_{\cO_\Fl} (\mathfrak{g}^0_0/\mathfrak{p}^0_0)^{-1} \stackrel{H}{\to} \cO^{\la}\otimes_{\cO_\Fl} (\mathfrak{g}^0_0/\mathfrak{p}^0_0)^{-1}\stackrel{+1}{\to}.\]
The claim follows by taking the cohomology and invoking the Kodaira vanishing.
\end{proof}

Denote by $\mathrm{fib}(\bar{d})$ the fiber of the morphism $\bar{d}: \cO^{\la,H=0}\to \cO^{\la,H=0}\otimes_{\cO_\Fl} (\mathfrak{g}^0_0/\mathfrak{p}^0_0)^{-1}$ in the bounded derived category of sheaves of abelian groups. There is a natural action of $\mathfrak{g}$ on it.
\begin{lem} \label{sl2trv}
$\mathfrak{sl}_2(C)$ acts trivially on $\mathrm{fib}(\bar{d})$.
\end{lem}

\begin{proof}
This follows from the construction of $\bar{d}$. Concretely, locally one can trivialize $(\mathfrak{g}^0_0/\mathfrak{p}^0_0)^{-1}$ and identify $\bar{d}$ with a non-zero element $u$ in $\mathfrak{sl}_2(C)$. Hence $u$ annihilates $\mathrm{fib}(\bar{d})$. The claim follows as $\mathfrak{sl}_2(C)$ is simple.
\end{proof}

\begin{prop} \label{stpM3}
$\mathfrak{sl}_2(C)$ acts trivially on $H^0(\Fl,\cO^{\la,H=0})$.
\end{prop}

\begin{proof}
Consider the exact sequence
\[H^0(\Fl, \mathrm{fib}(\bar{d}))\to H^0(\Fl,\cO^{\la,H=0})\stackrel{H^0(\bar{d})}{\to} H^0(\Fl,  \cO^{\la,H=0}\otimes_{\cO_\Fl} (\mathfrak{g}^0_0/\mathfrak{p}^0_0)^{-1}).\]
The last term vanishes by Proposition \ref{stpM21}, hence the first map is surjective and $\mathfrak{sl}_2(C)$-equivariant. The previous lemma implies the claim.
\end{proof}

\begin{proof}[Proof of Theorem \ref{MT} with $F=\Q$]
We can extend the scalars from $E$ to $C$. By the comparison result in Theorem \ref{thmLP}, it's enough to prove the corresponding result for $H^0(\Fl,\cO^{\la})$. This is a direct corollary of the previous proposition because as a corollary of Proposition \ref{stpM1}
$H^0(\Fl,\cO^{\la})=H^0(\Fl,\cO^{\la,H=0})$.
\end{proof}

\begin{rem}\label{dbark}
In the Hilbert case, we will also encounter $\cO^{\la,H=k}$ for $k\in\Z_{\geq 0}$. Note that $H=0$ on $\cO^{\la,H=k}\otimes_{\cO_\Fl} \cL_{(0,1)}^{-k}$, hence we get $\bar{d}$ on it as in \ref{bard}.
The usual BGG construction gives a $\mathfrak{sl}_2(C)$-equivariant morphism
\[\bar{d}^{k+1}: \cO^{\la,H=k}\to \cO^{\la,H=k}\otimes_{\cO_\Fl} \Omega_{\Fl}^{ k+1}, \]
and this complex is quasi-isomorphic to the complex
\[\bar{d}\otimes 1:\cO^{\la,H=k}\otimes_{\cO_\Fl} \cL_{(0,1)}^{-k}\otimes_C \Sym^{k} \mathrm{std} \to \cO^{\la,H=k}\otimes_{\cO_\Fl} \cL_{(0,1)}^{-k}\otimes \Omega^1_\Fl\otimes_C \Sym^{k} \mathrm{std}.\]
It follows that the action of $U(\mathfrak{sl}_2(C))$ on this complex in the derived category factors through $\End(\Sym^{k} \mathrm{std})$, which is finite-dimensional.
\end{rem}

\subsection{The general Hilbert case}
\begin{para}
We prove Theorem \ref{MT} for general totally real field in this subsection. The rough idea is to establish corresponding results in the previous subsection by induction.  We first need some notations. 

For $\tau\in\Sigma$, denote by $\Omega^1_{\{\tau\}}$ the equivariant line bundle on $\Fl$ associated to the weight of $\mathfrak{n}_\tau$ (explicitly, the weight has $k_\tau=-2$ and $k_{\tau'}=0$ for $\tau'\neq \tau$). As the notation suggests, it is isomorphic to the pullback of the cotangent bundle of $\Fl_\tau$. As in the case of modular curve, it's also isomorphic to the dual of $\mathfrak{g}_\tau^0/\mathfrak{p}^0_\tau$. For $I\subseteq \Sigma$ a subset and $\phi:I\to\Z$ a function, we define
\begin{itemize}
\item $\Omega^{\otimes \phi}_I:=\bigotimes_{\tau\in I}(\Omega^1_{\{\tau\}})^{\otimes \phi(\tau)}$ over $\cO_\Fl$; it is understood as $\cO_\Fl$ if $I$ is empty. 
\item $\Omega^{l}_I:=\bigotimes_{\tau\in I}(\Omega^1_{\{\tau\}})^{\otimes l}$ for $l\in\Z$.
\item $\mathrm{fib}(I,\phi)\in D^b(\cO_{\Fl})$ as the Koszul complex associated to $\{H_\tau-\phi(\tau)\}_{\tau\in I}$ on $\cO^{\la}$ \cite[\href{https://stacks.math.columbia.edu/tag/0623}{Tag 0623}]{stacks-project}. Our convention is that the complex starts at degree zero. For example, if $I=\{\tau\}$, then $\mathrm{fib}(I)$ is the fiber of $\cO^{\la}\stackrel{H_\tau-\phi(\tau)}{\longrightarrow}\cO^{\la}$. In particular, if $\tau'\in I$, we have an exact triangle 
\[\mathrm{fib}(I,\phi)\to \mathrm{fib}(I',\phi|_{I'}) \stackrel{H_{\tau'}-\phi(\tau')}{\longrightarrow} \mathrm{fib}(I',\phi|_{I'})\stackrel{+1}{\to}\] 
for $I'=I\setminus{\tau'}$.
\end{itemize}

Let $\lambda=(\{k_\tau\}_{\tau\in\Sigma};w)$ be a weight with $k_\tau\geq 0$ for all $\tau$. The $\mathfrak{g}$-representation $V_\lambda$ has a quotient of weight $\lambda$ when restricted to $\mathfrak{p}$, cf. Definition \ref{defnpw}. Consider the corresponding sequence of equivariant vector bundles on $\Fl$
\[0\to \mathcal{W}_\lambda \to V_\lambda\otimes_C \cO_\Fl \to \cL_\lambda \to 0\]
and its twist by $\cL_\lambda^{-1}$
\[0\to \mathcal{F}_\lambda \to V_\lambda\otimes_C \cL_\lambda^{-1}\to \cO_\Fl \to 0\]
where $\mathcal{F}_\lambda = \mathcal{W}_\lambda \otimes_{\cO_\Fl} \cL_\lambda^{-1}$.  By considering the weights, one sees that $\mathcal{F}_\lambda$ is a successive extension of line bundles and each such line bundle is isomorphic to $\Omega_I^{\otimes \phi}$ for some non-empty $I\subseteq \Sigma$, and $\phi:I\to \Z$ satisfying $1\leq\phi(\tau)\leq k_\tau$, $\tau\in I$. For example, let $\lambda_1$ be the parallel weight $(1;1)$. Then $V_{\lambda_1}$ is the tensor product of all the standard representations of $\mathfrak{g}_\tau$ and $\mathcal{F}_{\lambda_1}$ is a successive extension of line bundles isomorphic to $\Omega^1_I$ for some non-empty $I\subseteq \Sigma$.

\end{para}

\begin{prop}  \label{stpH2}
$H^{<d}(\Fl,\mathrm{fib}(I,\phi)\otimes_{\cO_{\Fl}} \mathcal{L}_\lambda^{-1})=0$ for any $I\subseteq \Sigma$, $\phi:I\to\Z$, and parallel weight $\lambda=(k;w)$ with $k>0$.
\end{prop}

\begin{proof}
This is the Kodaira vanishing in Theorem \ref{thmLP} if $I$ is empty. In general we use the triangle $\mathrm{fib}(I,\phi)\to \mathrm{fib}(I',\phi|_{I'}) \stackrel{H_{\tau'}-\phi(\tau')}{\longrightarrow} \mathrm{fib}(I',\phi|_{I'})\stackrel{+1}{\to}$ and argue by induction on $|I|$.
\end{proof}

\begin{prop} \label{stpM2}
Suppose $J\subseteq I\subseteq \Sigma$ and $\phi_I:I\to\Z$ is a function whose restriction to $J$ takes non-negative  values. Define $\phi:J\to\Z$ as $\phi(\tau)=\phi_I(\tau)+1$. Let $F=\mathrm{fib}(I,\phi_I)$. For $i<d$, there is a finite filtration on $H^i(\Fl,F\otimes \Omega^{\otimes \phi}_J)$ by $(\mathfrak{h}\times\mathfrak{g})$-invariant subspaces, such that
each graded piece is an eigenspace of $H_\tau$ for at least $\min(d-i,d-|J|)$ embeddings $\tau\in\Sigma\setminus J$ with non-negative integer eigenvalues.
\end{prop}

\begin{proof}
The proof will be similar to the proof of Proposition \ref{stpM1}. 


Here is the key technical lemma.

\begin{lem} \label{ind1}
For $i<d$ and subsets $I,J\subseteq \Sigma$ and functions $\phi_I:I\to\Z$, $\phi:J\to\Z_{\geq 1}$, there is a finite $(\mathfrak{h}\times\mathfrak{g})$-invariant  filtration on 
\[H^i(\Fl,F\otimes_{\cO_\Fl}\Omega^{\otimes \phi}_J)\] 
with $F=\mathrm{fib}(I,\phi_I)$, such that each non-zero graded piece can be embedded into a subquotient of 
\[H^{i+1}(\Fl,F\otimes_{\cO_\Fl}\Omega^{\otimes \phi'}_{J\cup J'})\]
as a  $(\mathfrak{h}\times\mathfrak{g})$-module for some  subset $J'$ of $\Sigma$ disjoint from $J$ and $\phi':J\cup J'\to\Z_{\geq 1}$ satisfying that $\phi'(\tau)\geq\phi(\tau)$ for $\tau\in J$ and either 
\begin{itemize}
\item $J'$ is non-empty; or
\item  $\phi'(\tau')>\phi(\tau')$ for some $\tau'\in J$.
\end{itemize}
In short the pair $(J,\phi)\neq (J\cup J',\phi')$.
\end{lem}

We would like to  keep applying this lemma. The following lemma provides a sufficient condition to rule out the second possibility in Lemma \ref{ind1}.

\begin{lem} \label{lemevH}
Same setup as in Lemma \ref{ind1}. Let  $M$ be a subquotient of $H^i(\Fl,F\otimes\Omega^{\otimes \phi}_J)$ as a $(\mathfrak{h}\times\mathfrak{g})$-module that can be embedded into a subquotient of $H^j(\Fl,F\otimes\Omega^{\otimes \phi'}_{J\cup J'})$ for some $j$, $J'$ disjoint from $J$, and $\phi':J\cup J'\to \Z_{\geq 1}$. Then $H_\tau=\phi'(\tau)-1$ on $M$ for $\tau\in J'$.
\end{lem}

\begin{proof}
By Theorem \ref{HCas}, 
\begin{itemize}
\item $\Omega_\tau=\frac{1}{2}(H_\tau+1)^2-\frac{1}{2}$ on $F\otimes\Omega^{\otimes \phi}_J$,
\item $\Omega_\tau=\frac{1}{2}(H_\tau+1-2\phi'(\tau))^2-\frac{1}{2}$ on $F\otimes\Omega^{\otimes \phi'}_{J\cup J'}$.\end{itemize}
Hence $\frac{1}{2}(H_\tau+1)^2-\frac{1}{2}=\frac{1}{2}(H_\tau+1-2\phi'(\tau))^2-\frac{1}{2}$ on $M$. Thus $H_\tau=\phi'(\tau)-1$ on $M$ because $\phi'(\tau)\geq 1$.
\end{proof}

Assuming Lemma \ref{ind1}, we complete the proof of Proposition \ref{stpM2}.  Keep applying Lemma \ref{ind1}. We get a finite $(\mathfrak{h}\times\mathfrak{g})$-invariant filtration on $H^i(\Fl,F\otimes \Omega^{\otimes \phi}_J)$ such that each non-zero graded piece $M$ can be  embedded as a  $(\mathfrak{h}\times\mathfrak{g})$-module  into subquotients of $H^{i+1}(\Fl,F\otimes \Omega_{J_1}^{\otimes \phi_1})$, $H^{i+2}(\Fl,F\otimes \Omega_{J_2}^{\otimes \phi_2})\cdots$ for subsets $J\subseteq J_1\subseteq J_2\subseteq\cdots \subseteq\Sigma$ and $\phi_i:J_i\to\Z_{\geq 1}$. We claim that all the inclusions $J\subseteq J_1 \subseteq J_2 \subseteq \cdots$ are strict. If not, say $J_j=J_{j+1}$ with $J_0$ understood as $J$ and $\phi_0=\phi$. Then for $\tau\in J_j$, we claim that both $\phi_j(\tau)-1$ and $\phi_{j+1}(\tau)-1$ are equal to the eigenvalue of $H_\tau$ on $M$. (Recall that  $M$ is a subquotient of $H^i(\Fl,F\otimes \Omega^{\otimes \phi}_J)$  by our assumption.)
\begin{itemize}
\item If $\tau\notin J$, this is Lemma \ref{lemevH},
\item If $\tau\in J$, this follows from Lemma \ref{varlemevH} below. 
\end{itemize}
Hence $\phi_j=\phi_{j+1}$ which contradicts Lemma \ref{ind1}. Thus the sequence $J\subseteq J_1\subseteq J_2\subseteq\cdots \subseteq\Sigma$ is strict increasing. When it terminates,   $M$ is embedded either into $H^*(\Fl,F\otimes \Omega_{\Sigma}^{\otimes \phi_\Sigma})$ for some $\phi_\Sigma:\Sigma\to\Z_{\geq 1}$ or  into $H^d(\Fl,F\otimes\Omega_{J\cup J'}^{\otimes \phi'})$ for some $J'$ disjoint from $J$ satisfying $|J'|\geq d-i$  and $\phi':J\cup J'\to \Z_{\geq 1}$. In both cases Lemma \ref{lemevH} gives the desired result.
\end{proof}

\begin{lem} \label{varlemevH}
Same setup as in Proposition \ref{stpM2}. Let $J'\subseteq \Sigma$ and $\phi':J'\to\Z$. Let $\tau\in J$. The Casimir operator $\Omega_\tau$ acts via  $\frac{1}{2}(\phi(\tau)+2\phi'(\tau))^2-\frac{1}{2}$ on $F\otimes_{\cO_\Fl}\Omega^{\otimes \phi}_{J}\otimes_{\cO_\Fl}\Omega^{\otimes \phi'}_{J'}$. In particular, $F\otimes\Omega^{\otimes \phi}_{J}$ and $F\otimes\Omega^{\otimes \phi}_{J}\otimes\Omega^{\otimes \phi'}_{J'}$ have different eigenvalues of $\Omega_\tau$  if $\phi'(\tau)>0$.
\end{lem}

\begin{proof}
This is a direct consequence of Theorem \ref{HCas} since
$H_\tau$ acts via $\phi(\tau)-1$ on $F$. 
\end{proof}

\begin{proof}[Proof of Lemma \ref{ind1}]
First assume that $J$ is empty. Consider the tensor product of $F$ and the sequence 
\[0\to \mathcal{F}_{\lambda_1} \to V_{\lambda_1}\otimes_C \cL_{\lambda_1}^{-1}\to \cO_\Fl \to 0\]
and its cohomology. We get an exact sequence (we omit all subscripts $\cO_\Fl$ in tensor products)
\[H^i(\Fl, F\otimes\cL_{\lambda_1}^{-1})\otimes_C V_{\lambda_1}\to H^i(\Fl, F) \to H^{i+1}(\Fl, F\otimes\mathcal{F}_{\lambda_1}). \]
The first term is zero  by Proposition \ref{stpH2}. This means that the second map in the above exact sequence is injective. Since $\mathcal{F}_{\lambda_1}$ is filtered by $\Omega^1_{J'}$ with $J'$ non-empty, we see that $H^i(\Fl, F)$ has a finite $(\mathfrak{h}\times\mathfrak{g})$-invariant  filtration whose graded pieces can be  embedded into a subquotient of $H^{i+1}(\Fl,F\otimes \Omega^1_{J'})$ for some non-empty $J'$. 

Next assume that $J$ is non-empty.
Set $k=\max_{\tau\in J}\phi(\tau)$ and $\lambda=(\{k_\tau\}_{\tau\in\Sigma};0)$ with $k_\tau=2k-2\phi(\tau)$ if $\tau\in J$ and $k_\tau=2k$ otherwise.
Consider the tensor product of $F\otimes_{\cO_\Fl}\Omega^{\otimes \phi}_J$ and the sequence 
\[0\to \mathcal{F}_\lambda \to V_\lambda \otimes_C \cL_\lambda^{-1}\to \cO_\Fl \to 0\]
and its cohomology. We get an exact sequence 
\[H^i(\Fl, F\otimes\Omega^{\otimes \phi}_J\otimes\cL_{\lambda}^{-1})\otimes_C V_\lambda \to H^i(\Fl, F\otimes\Omega^{\otimes \phi}_J) \to H^{i+1}(\Fl, F\otimes\Omega^{\otimes \phi}_J\otimes\mathcal{F}_\lambda). \]
Note that  $(\Omega^{\otimes \phi}_J\otimes\cL_{\lambda}^{-1})^{-1}$ is of parallel weight $(2k;0)$. Hence $H^i(\Fl, F\otimes\Omega^{\otimes \phi}_J\otimes\cL_{\lambda}^{-1})=0$ by Proposition \ref{stpH2} and the second map in the above exact sequence is injective.
We complete the proof by noting that $\mathcal{F}_\lambda$ is either $0$ or a successive extension of line bundles isomorphic to $\Omega_I^{\otimes \phi}$ for some non-empty $I\subseteq \Sigma$, and $\phi:I\to \Z_{\geq 1}$.
\end{proof}

\begin{prop} \label{propSigmavan}
Suppose $\phi:\Sigma\to\Z_{\geq 1}$. For $i<d$, 
\[H^i(\Fl,\mathrm{fib}(\Sigma,\phi-1)\otimes_{\cO_\Fl} \Omega^{\otimes \phi}_\Sigma)=0.\] 
\end{prop}

\begin{proof}
Let $F=\mathrm{fib}(\Sigma,\phi-1)$ and $k=\max_{\tau\in \Sigma}\phi(\tau)$. Let $\lambda$ be the weight with $k_\tau=2k-2\phi(\tau)$ and $w=0$, and $\lambda_{2k}$ be the parallel weight $(2k;0)$. Consider $F\otimes_{\cO_\Fl}\cL_{\lambda_{2k}}^{-1}\otimes_C V_\lambda$. It has a finite filtration whose filtered pieces are of the form $F\otimes\Omega_{\Sigma}^{\otimes \phi'}$, where $\phi':\Sigma\to \Z$ satisfies $\phi'(\tau)\in \{\phi(\tau),\phi(\tau)+1,\cdots,2k-\phi(\tau)\}$.  Lemma \ref{varlemevH} implies that the common eigenvalues of $\Omega_\tau$, $\tau\in\Sigma$ on $F\otimes  \Omega^{\otimes \phi}_\Sigma$ are distinct from the eigenvalues on other factors. Therefore $F\otimes  \Omega^{\otimes \phi}_\Sigma$ is a direct summand of $F\otimes_{\cO_\Fl}\cL_{\lambda_{2k}}^{-1}\otimes_C V_\lambda$. Since $F\otimes_{\cO_\Fl}\cL_{\lambda_{2k}}^{-1}\otimes_C V_\lambda$ has no cohomology below degree $d$ by Proposition \ref{stpH2}, the same holds for $F\otimes  \Omega^{\otimes \phi}_\Sigma$.
\end{proof}

\begin{cor} \label{corSigmavan}
Let $J$ be a subset of $\Sigma$. Suppose $\phi_\Sigma:\Sigma\to\Z$ is a function with non-negative values on $J$. Let $\phi:J\to \Z_{\geq 1}$ be $\phi(\tau)=\phi_\Sigma(\tau)+1$ and $F=\mathrm{fib}(\Sigma,\phi_\Sigma)$. For $i<|J|$,
\[H^i(\Fl,F\otimes_{\cO_\Fl} \Omega^{\otimes \phi}_J)=0.\] 
\end{cor}

\begin{proof}
We argue by backward induction on $|J|$. The base case with $J=\Sigma$ is Proposition \ref{propSigmavan}. In general, note that $H_\tau$ acts on $F\otimes_{\cO_\Fl} \Omega^{\otimes \phi}_J$ as the scalar $\phi_\Sigma(\tau)$ for $\tau\in\Sigma\setminus J$. By Lemma \ref{ind1}, we can find a finite filtration on $H^i(\Fl,F\otimes_{\cO_\Fl} \Omega^{\otimes \phi}_J)$ whose non-zero graded pieces can be embedded into a subquotient of $H^{i+1}(\Fl,F\otimes_{\cO_\Fl}\Omega^{\otimes \phi'}_{J\cup J'})$ for some $J'$ and $\phi':J\cup J'\to\Z_{\geq 1}$. By exactly the same argument as in the proof of Proposition \ref{stpM2} (using Lemma \ref{lemevH} and Lemma \ref{varlemevH}), 
 we see that every $J'$ and $\phi'$ showing up here satisfy that 
\begin{enumerate}
\item $J'$ is non-empty;
\item $\phi'(\tau)=\phi_\Sigma(\tau)+1$ for $\tau\in J\cup J'$.
\end{enumerate}
Thus $H^{i+1}(\Fl,F\otimes_{\cO_\Fl}\Omega^{\otimes \phi'}_{J\cup J'})=0$  by the induction hypothesis and proves the induction step.
\end{proof}

\begin{defn}
For $i\geq 0$, we say a $\mathfrak{g}$-module satisfies \textit{(i-SL)} if it has a finite filtration by $\mathfrak{g}$-submodules and each graded piece is annihilated by an ideal of finite codimension in $U(\mathfrak{sl}_{2,\tau})$ for at least $(d-i)$ $\tau$'s. This property holds when passing to a $\mathfrak{g}$-subquotient.
\end{defn}

\begin{rem} \label{remfincodid}
For $k\geq 0$, we let $I_k=\ker\left(U(\mathfrak{sl}_{2}(C))\to \End(\Sym^{k} \mathrm{std})\right)$. It's easy to see that an ideal of finite codimension in $U(\mathfrak{sl}_{2}(C))$ must contain $I_0^N I_1^N \cdots I_N^N$ for some $N$. 
\end{rem}

\begin{prop}
Suppose $J\subseteq I\subseteq \Sigma$ and $\phi_I:I\to\Z_{\geq 0}$. Define $\phi:J\to\Z_{\geq 1}$ as $\phi(\tau)=\phi_I(\tau)+1$. Let $F=\mathrm{fib}(I,\phi_I)$. Then $H^{i}(\Fl,F\otimes_{\cO_\Fl}\Omega^\phi_J)$ satisfies \textit{(i-SL)}.
\end{prop}

\begin{proof}
We may assume $i<d$. We will do backward induction on $|I|+|J|$. The rough idea is as follows. By applying Lemma \ref{ind1} enough times,  one may assume $|I|-|J|\geq d-i$. Then one can show that a vector in $H^{i}(\Fl,F\otimes_{\cO_\Fl}\Omega^\phi_J)$ is either killed by a finite-codimensional ideal of $U(\mathfrak{sl}_{2,\tau})$ for $\tau \in I\setminus J$ or one can apply the induction hypothesis. 

The base case with $J=I=\Sigma$ is clear by Proposition \ref{propSigmavan}. Assume the proposition is proved for $|I|+|J|\geq k+1$ and consider the case when $|I|+|J|=k$. Suppose that  $d-i>|I|-|J|$. We separate two cases:
\begin{enumerate}
\item If $I=\Sigma$, this follows from Corollary \ref{corSigmavan}.
\item If $I\neq \Sigma$, then $|I|-|J| < \min( d-i, d-|J|)$. For each graded piece in Proposition \ref{stpM2}, there exists $\tau\notin I$ such that $H_\tau$ acts on it by a non-negative integer. Hence $H^{i}(\Fl,F\otimes_{\cO_\Fl}\Omega^\phi_J)$ is annihilated by 
\[\Phi:=\prod_{\tau\notin I}H_\tau^N(H_\tau-1)^N\cdots (H_\tau-N)^N\]
for some sufficiently large $N$. Hence if we let $\mathrm{fib}(\Phi)$ be the fiber of the morphism $\Phi: F\to F$, the natural map
\[H^i(\Fl, \mathrm{fib}(\Phi)\otimes \Omega^\phi_J)\to H^i(\Fl, F\otimes_{\cO_\Fl}\Omega^\phi_J)\]
is surjective. Thus it's enough to show that $H^i(\Fl, \mathrm{fib}(\Phi)\otimes\Omega^\phi_J)$ satisfies \textit{(i-SL)}. This follows from the induction hypothesis because by the octahedron axiom of derived category, $H^i(\Fl, \mathrm{fib}(\Phi)\otimes\Omega^\phi_J)$ is filtered by subquotients of $H^i(\Fl, \mathrm{fib}(I',\phi_{I'})\otimes\Omega^\phi_J)$ for $I'=I\cup\{\tau\}$ containing $I$ strictly and $\phi_{I'}:I'\to\Z_{\geq 0}$ extending $\phi_I$.
\end{enumerate}

We may assume $|I|-|J|\geq d-i$ from now on. For each $\tau\in I\setminus J$, let $k_\tau=\phi_I(\tau)+1$. As in \ref{bard} and \ref{dbark}, the ($k_\tau$-th iteration of) action of $\cO_{\Fl}\otimes_C\mathfrak{sl}_{2,\tau}\subseteq \mathfrak{g}_\tau^0$ on $F$ induces  a morphism 
\[\bar{d}_\tau^{k_\tau}: F\otimes_{\cO_\Fl}\Omega^\phi_J \to F\otimes_{\cO_\Fl}\Omega^{k_\tau}_{\{\tau\}}  \otimes_{\cO_\Fl}\Omega^\phi_{J}, \]
and $\mathrm{fib}(\bar{d}_\tau^{k_\tau})$ is annihilated by the ideal $\ker\left(U(\mathfrak{sl}_{2,\tau}(C))\to \End(\Sym^{k_\tau-1} \mathrm{std})\right)$.
 Consider the sum of all such $H^i(\bar{d}^{k_\tau}_\tau)$'s:
\[H^i(\Fl, F\otimes_{\cO_\Fl}\Omega^\phi_{J})\stackrel{H^i(\bar{d}^{k_\tau}_\tau)}{\longrightarrow} \bigoplus_{\tau\in I\setminus J} H^i(\Fl, F\otimes_{\cO_\Fl}\Omega^\phi_{J}\otimes \Omega^{k_\tau}_{\{\tau\}}).\]
By our induction hypothesis, all $H^i(\Fl, F\otimes_{\cO_\Fl}\Omega^\phi_{J}\otimes \Omega^{k_\tau}_{\{\tau\}})$ satisfy \textit{(i-SL)}. On the other hand, the kernel is annihilated by all  $\ker\left(U(\mathfrak{sl}_{2,\tau}(C))\to \End(\Sym^{k_\tau-1} \mathrm{std})\right)$, $\tau\in I\setminus J$, hence also satisfies  \textit{(i-SL)} because $|I\setminus J|\geq d-i$. This finishes the induction step.
\end{proof}

\begin{proof}[Proof of Theorem \ref{MT}]
By Theorem \ref{thmLP}, $\tilde{H}^i(K^p)_{E}^{\la}\widehat\otimes_E C\cong H^i(\Fl,\cO^{\la})$ which satisfies  \textit{(i-SL)} by the previous proposition with both $I$ and $J$ being empty.
\end{proof}

\section{Applications}
\subsection{Dimension Theory}
\begin{para}
We first recall the dimension theory for admissible $p$-adic representations and will
discuss several applications of the main results in the next subsection. Our reference is \cite{ST03}, \cite[Section 2, 3]{DPS23}.

If $A$ is an associative, unital and Noetherian (not necessarily commutative) ring and $M$ a finitely generated $A$-module, we denote by $E^q_A(M)=\Ext^q_A(M,A)$ and 
\[j_A(M)=\inf \{q| E^q_A(M)\neq 0\}\in \Z_{\geq 0}\cup \{\infty\}\]
the \textit{grade} of $M$. Clearly if there is a short exact sequence of finitely generated $A$-modules $0\to M'\to M\to M''\to 0$, then
\begin{eqnarray} \label{jadd}
j_A(M)\geq \min(j_A(M'),j_A(M'')).
\end{eqnarray}
This is an equality if $A$ is Auslander-Gorenstein \cite[Proposition 3.2]{DPS23}, which will be the case in our application. If $B$ is another associative, unital and Noetherian ring and there is a flat unital ring homomorphism $\phi: A\to B$, then 
\[j_B(M\otimes_{A,\phi} B) \geq j_A(M).\]
Indeed note that there is a natural right $A$-module structure on $E^q_A(M)$. By taking a projective resolution of $M$, it's easy to see that $E^q_{B}(M\otimes_{A,\phi} B)\cong E^q_A(M)\otimes_{A,\phi} B$.

If $K$ is a compact $p$-adic Lie group, we denote its Iwasawa algebra by $\Z_p[[K]]$ and denote $\F_p[[K]]:=\Z_p[[K]]/(p), \Q_p[[K]]:=\Z_p[[K]]\otimes_{\Z_p}\Q_p$. Moreover if $K$ is uniform pro-$p$, for $r_n=p^{-1/p^n}$, $n\geq 0$, we denote by $D(K,\Q_p)$ its distribution algebra and
$D_{r_n}(K,\Q_p)$ the $r_n$-distribution algebra, cf. \cite[Section 4]{ST03}.  These rings are Auslander regular when $n\geq 1$ by \cite[Theorem 8.9]{ST03}. In particular they are Auslander-Gorenstein.

If $X$ is a topological $\Q_p$-vector space, we let $X^*$ be its topological $\Q_p$-dual. 
\end{para}

\begin{defn} 
Let $G$ be a locally analytic $\Q_p$-analytic group.
\begin{enumerate}
\item For an admissible Banach representation $\Pi$ of $G$ over $\Q_p$, we set
\[d(\Pi)=\dim K-j_{\Q_p[[K]]}(\Pi^*),\]
where  $K$ is an open compact subgroup of $G$. This number $d(\Pi)$
 is independent of $K$. (To see this, suppose $K'\subseteq K$ is an open normal subgroup one can apply  \cite[Lemma 8.8]{ST03} to $\Q_p[[K']]\subseteq \Q_p[[K]]$.)
\item For an admissible locally analytic representation $\Pi$ of $G$ over $\Q_p$, we set
\[d(\Pi)=\dim K- \min_n j_{D_{r_n}(K,\Q_p)}(D_{r_n}(K,\Q_p)\otimes_{D(K,\Q_p)}\Pi^*)\]
where  $K$ is a uniform pro-p open subgroup of $G$ (whose existence is a classical result of Lazard). Again $d(\Pi)$ is independent of $K$. See \cite[Remark 8.10]{ST03}.
\end{enumerate}
\end{defn}

\begin{lem} \label{comBladim}
 $d(\Pi)=d(\Pi^{\la})$ if $\Pi$ is an admissible Banach representation $\Pi$ of $G$ over $\Q_p$.
\end{lem}
\begin{proof}
\cite[Lemma 3.9]{DPS23}.
\end{proof}

\begin{rem}
If $\Pi$ is an admissible Banach space representation of $G$ over $\Q_p$, then $d(\Pi)$ agrees with the Gelfand-Kirillov dimension of its  $\mod p$ representation: suppose $H$ is a uniform pro-$p$ open subgroup of $G$ and $\Pi^o$ is a $H$-invariant lattice, then 
\[d(\Pi)= \dim H \lim_{n\to \infty} \frac{\log \dim_{\F_p}(\Pi^o/p)^{H_n}}{\log[H:H_n]}\]
where $H_n=H^{p^n}:=\{h^{p^n},h\in H\}$ is an open subgroup of $H$. See  \cite[Section 3.4]{DPS23}. From this point of view $j_A$ is also called codimension.
\end{rem}

Let $E$ be a finite extension of $\Q_p$. Let $H$ be a uniform pro-$p$ open subgroup of $G$. As in \cite[Section 2.2]{DPS23} there is a $\cO_E$-sub-Lie algebra $\mathfrak{g}_{H,\cO_E}$ of $\Lie(H)\otimes_{\Q_p} E$ which is finite free of rank $\dim H$ over $\cO_E$. The classical Lazard isomorphism provides natural isomorphisms for $n\geq 0$,
\[D_{1/p}(H^{p^n},\Q_p)\otimes_{\Q_p} E\cong U(\widehat{p^n\mathfrak{g}_{H,\cO_E}})[1/p]\]
where $U(\widehat{p^n\mathfrak{g}_{H,\cO_E}})$ denotes the $p$-adic completion of the universal enveloping algebra $U(p^n\mathfrak{g}_{H,\cO_E})$. By considering the $p$-adic filtration $p^\bullet U(p^n\mathfrak{g}_{H,\cO_E})$ on $U(p^n\mathfrak{g}_{H,\cO_E})[1/p]=U(\Lie(H))\otimes_{\Q_p} E=: U(\Lie(H))_E$ and invoking \cite[Proposition 1.2]{ST03}, we deduce that $U(\widehat{p^n\mathfrak{g}_{H,\cO_E}})[1/p]$ is flat over $U(\Lie(H))_E$.
Moreover \cite[Proposition 2.3]{DPS23} says that $D_{r_n}(H,\Q_p)$ is a natural free (both left and right) module over $D_{1/p}(H^{p^n},\Q_p)$. 

\begin{prop} \label{proplbdj}
$D_{r_n}(H,E):=D_{r_n}(H,\Q_p)\otimes_{\Q_p} E$ is a flat module of $U(\Lie(H))_E$. In particular, if $I$ is two-sided ideal of $U(\Lie(H))_E$ and $M$ is a finitely generated $D_{r_n}(H,E)/(I)$-module, then
\[j_{D_{r_n}(H,E)}(M)\geq j_{U(\Lie(H))_E}(U(\Lie(H))_E/I).\]
\end{prop}

\begin{proof}
The flatness follows from the above discussion. It also implies that 
\[j_{D_{r_n}(H,E)}(D_{r_n}(H,E)/(I))\geq j_{U(\Lie(H))_E}(U(\Lie(H))_E/I).\]
Since $D_{r_n}(H,E)$ is Auslander-Gorenstein and $M$ is a quotient of $(D_{r_n}(H,E)/(I))^{\oplus d}$ for some $d$, we have 
\[j_{D_{r_n}(H,E)}(M)\geq j_{D_{r_n}(H,E)}(D_{r_n}(H,E)/(I)).\]
\end{proof}

In our application, $I$ has the following special form.

\begin{lem} \label{lemj}
Suppose that there is an isomorphism of $E$-Lie algebras $\Lie H\otimes_{\Q_p} E\cong L\times L'$. Let $I$ be a two-sided ideal of $U(L)$ of finite codimension. 
\[j_{U(\Lie(H))_E}(U(\Lie(H))_E/(I))=j_{U(L)}(U(L)/I)=\dim_E L.\]
\end{lem}

\begin{proof}
The first equality is clear as $U(\Lie(H))_E=U(L)\otimes_E U(L')$. For the second equality, consider the PBW filtration on $U(L)$ whose graded algebra is the symmetric algebra $S(L)$ by the PBW theorem. By a theorem of Bjork (cf. \cite[Theorem 3.4]{DPS23}), we can take a good filtration on $U(L)/I$ and have
\[j_{U(L)}(U(L)/I)=j_{S(L)}(\mathrm{Gr}(U(L)/I)).\]
Since $\mathrm{Gr}(U(L)/I)$ is a finite-dimensional $E$-vector space by our assumption, we easily reduce to the case where $\mathrm{Gr}(U(L)/I)=S(L)/J$ for a maximal ideal $J$ of $S(L)$. The standard Koszul complex shows that the grade is $\dim L$.
\end{proof}

\begin{para}
Now we specialize to the Hilbert case. Fix a level $K^p$ away from $p$ and a level $K_p$ at $p$. The action of $K_p$ on the tower of Shimura varieties $\{\Sh_{K^pK'_p,\bC}\}_{K'_p\subseteq \rmG(\Q_p)}$ is not faithful in general and factors through a quotient 
$p$-adic Lie group denoted by $\widetilde{K_p}$. Hence $\widetilde{K_p}$ acts on the completed cohomology. We note that the kernel of this quotient map is  a closed subgroup in the center of $\rmG(\Q_p)$ whose dimension is described by Leopoldt's conjecture. In particular, $\Lie \widetilde{K_p}= \mathfrak{sl}_2(F_p)\times \Lie Z( \widetilde{K_p})$.
\end{para}

\begin{thm} \label{mtj}
For $i\geq 0$,
\[d(\tilde{H}^i(K^p)_{\Q_p})\leq \dim \widetilde{K_p}-3(d-i).\]
Equivalently $j_{\Q_p[[\widetilde{K_p}]]}(\tilde{H}^i(K^p)_{\Q_p})\geq 3(d-i)$.
\end{thm}

\begin{proof}
Fix a uniform pro-$p$ open subgroup $H$ of $\widetilde{K_p}$. By Lemma \ref{comBladim}, it suffices to prove that $n\geq 1$,
\[j_{D_{r_n}(H,\Q_p)}(D_{r_n}(H,\Q_p)\otimes_{D(H,\Q_p)}\Pi^*)\geq  3(d-i)\]
where $\Pi=\tilde{H}^i(K^p)_{\Q_p}^{\la}$. Fix a finite extension $E$ of $\Q_p$ containing  all embeddings of $F$. Let $M=D_{r_n}(H,\Q_p)\otimes_{D(H,\Q_p)}\Pi^*$ and $M_E=M\otimes_{\Q_p} E$. It's easy to see (\cite[Lemma 8.8]{ST03}) that $j_{D_{r_n}(H,\Q_p)}(M)=j_{D_{r_n}(H,E)}(M_E)$. By Theorem \ref{MT},  $M$ is filtered by $\Lie(\widetilde{K_p})_E$-modules on which the action of $U(\mathfrak{sl}_{2,\tau_1}(E)\times\cdots\times \mathfrak{sl}_{2,\tau_{d-i}}(E))$ factors through a finite-dimensional quotient for some embeddings $\tau_1,\cdots,\tau_{d-i}$. Thus we can apply Proposition \ref{proplbdj} and Lemma \ref{lemj} to get the desired lower bound.\end{proof}

\subsection{A density result}
\begin{para}
Completed cohomology with compact support is defined as
\[
\tilde{H}_c^i(K^p):=\varprojlim_n\varinjlim_{K_p\subseteq \rmG(\Q_p)} H_c^i(\Sh_{K^pK_p,\bC},\Z/p^n).
\]
Again $\tilde{H}_c^i(K^p)_{\Q_p}:=\tilde{H}_c^i(K^p)\otimes_{\Z_p}{\Q_p}$ is an admissible representation of $\widetilde{K_p}$. 
\end{para}

\begin{thm} \label{mtref}
Suppose $K^pK_p$ is neat. There are natural isomorphisms
\[\Hom_{\Q_p[[\widetilde{K_p}]]}(\tilde{H}^d(K^p)_{\Q_p}^*,\Q_p[[\widetilde{K_p}]])=\tilde{H}_c^d(K^p)_{\Q_p}^*,\]
\[\Hom_{\Q_p[[\widetilde{K_p}]]}(\tilde{H}_c^d(K^p)_{\Q_p}^*,\Q_p[[\widetilde{K_p}]])=\tilde{H}^d(K^p)_{\Q_p}^*.\]
Particularly both $\tilde{H}_c^d(K^p)_{\Q_p}^*$ and $\tilde{H}^d(K^p)_{\Q_p}^*$ are reflexive $\Q_p[[\widetilde{K_p}]]$-modules.
\end{thm}

\begin{proof}
Following \cite[\S1.1]{CE12}, we define the completed homology and its Borel-Moore variant
\[\tilde{H}_i(K^p):=\varinjlim_{K_p\subseteq \rmG(\Q_p)} H_i(\Sh_{K^pK_p,\bC},\Z/p),\]
\[\tilde{H}_i^{BM}(K^p):=\varinjlim_{K_p\subseteq \rmG(\Q_p)} H_i^{BM}(\Sh_{K^pK_p,\bC},\Z/p),\]
where $H^{BM}_*$ denotes the Borel-Moore homology. There are natural isomorphisms
\[\tilde{H}_i(K^p)\otimes_{\Z_p}\Q_p\cong \tilde{H}^i(K^p)_{\Q_p}^*,\]
\[\tilde{H}_i^{BM}(K^p)\otimes_{\Z_p}\Q_p\cong \tilde{H}_c^i(K^p)_{\Q_p}^*\]
by Theorem 1.1.3 in loc.cit. and noting that the torsion submodules of $\tilde{H}_c^i(K^p)$ and  $\tilde{H}^i(K^p)$ have bounded exponents by Theorem 1.1.1 in loc.cit.. Now inverting $p$ in the Poincar\'e duality sequences in \S 1.3 loc.cit., we get two spectral sequences
\begin{eqnarray*} 
E_{2}^{i,j}=E^i_{\Q_p[[\widetilde{K_p}]]}(\tilde{H}^j(K^p)_{\Q_p}^*)\Longrightarrow \tilde{H}_c^{2d-i-j}(K^p)_{\Q_p}^* ,\\
E_{2}^{i,j}=E^i_{\Q_p[[\widetilde{K_p}]]}(\tilde{H}^j_c(K^p)_{\Q_p}^*)\Longrightarrow \tilde{H}^{2d-i-j}(K^p)_{\Q_p}^*
\end{eqnarray*}
(The $d$ in the reference refers to the dimension as a real manifold so is equal to $2d$ in our case.) 
Note that $E_2^{0,d}=\Hom_{\Q_p[[\widetilde{K_p}]]}(\tilde{H}^d(K^p)_{\Q_p}^*,\Q_p[[\widetilde{K_p}]]),\Hom_{\Q_p[[\widetilde{K_p}]]}(\tilde{H}_c^d(K^p)_{\Q_p}^*,\Q_p[[\widetilde{K_p}]])$.
Therefore it is enough to prove the following lemma.
\end{proof}

\begin{lem}
For $i<d$,
\begin{eqnarray} \label{cdv1}
E^{d-i}_{\Q_p[[\widetilde{K_p}]]}(\tilde{H}^{i}(K^p)_{\Q_p}^*)=E^{d-i+1}_{\Q_p[[\widetilde{K_p}]]}(\tilde{H}^{i}(K^p)_{\Q_p}^*)=0.\\ \label{cdv2}
E^{d-i}_{\Q_p[[\widetilde{K_p}]]}(\tilde{H}_c^{i}(K^p)_{\Q_p}^*)=E^{d-i+1}_{\Q_p[[\widetilde{K_p}]]}(\tilde{H}_c^{i}(K^p)_{\Q_p}^*)=0.
\end{eqnarray}
\end{lem}

\begin{proof}
\eqref{cdv1} follows from Theorem \ref{mtj} because $d-i+1<3(d-i)$. For \eqref{cdv2}, since the Hilbert modular variety is non-compact, $\tilde{H}^{0}_c(K^p)_{\Q_p}^*=0$. Thus we may assume $i\geq 1$ and $d\geq 2$ in the rest.

As explained in \cite[\S 1.5]{CE12}, there is long exact sequence
\[\cdots \tilde{H}^{\bullet -1}(\partial) \to \tilde{H}^{\bullet}_c \to \tilde{H}^{\bullet} \to \tilde{H}^{\bullet }(\partial)\to \cdots\]
where $\tilde{H}^{\bullet }(\partial)$ is the boundary completed cohomology and can be described using Borel-Serre compactification. Specialized to the Hilbert case, the formula (1.4) in the reference says that as a representation $\rmG(\Q_p)$, 
\[\tilde{H}^{\bullet }(\partial)\cong \mathrm{Ind}^{\rmG(\Q_p)}_{\mathrm{B}(\Q_p)} \tilde{H}^{\bullet}_{\mathrm{T}}\]
where $\mathrm{B}$ denotes the restriction of scalars of the upper triangular Borel subgroup from $F$ to $\Q$, $ \mathrm{Ind}$ denotes the continuous induction, $\tilde{H}^{\bullet}_{\mathrm{T}}$ is the completed cohomology (of certain tame level) associated to the diagonal torus $\mathrm{T}\cong (\Res_{F/\Q}\GL_1)^2$ of $\mathrm{B}$ and $\mathrm{B}(\Q_p)$ acts on it via the quotient map to $\mathrm{T}(\Q_p)$. Now $d-i+1\leq d$. From the long exact sequence, it suffices to show
\[E^{\leq d}_{\Q_p[[\widetilde{K_p}]]}(\mathrm{Ind}^{\rmG(\Q_p)}_{\mathrm{B}(\Q_p)} \tilde{H}^{j}_{T,\Q_p})=0,\]
for any $j$, equivalently $j_{\Q_p[[\widetilde{K_p}]]}(\mathrm{Ind}^{\rmG(\Q_p)}_{\mathrm{B}(\Q_p)} \tilde{H}^{j}_{T,\Q_p})\geq d+1$.

Let $H$ be the subgroup of $(\cO_{F,p}^\times)^2\subseteq (F_p^\times)^2\cong \mathrm{T}(\Q_p)$ acting trivially on the completed cohomology $\tilde{H}^{j}_{\mathrm{T}}$. Then $H$ contains $Z_0:=\ker(\GL_2(\cO_{F,p})\to \widetilde{\GL_2(\cO_{F,p})})$ and an open subgroup of $(\overline{\cO_F^\times})^2$ (depending on the level), 
where $\overline{\cO_F^\times}$ is the closure of $\cO_F^\times$ in $\cO_{F,p}^\times$.
Denote by $d_0=\dim \overline{\cO_F^\times}$ as a $p$-adic Lie group.
Since we assume $d\geq 2$, $\cO_F^\times$ is not finite by Dirichlet's unit theorem, hence $d_0\geq 1$.
(Leopoldt's conjecture predicts that $d_0=d-1$.) Consider the following subgroup of $\mathrm{B}(\Q_p)$
\[
J=\left \{ \begin{pmatrix} a & n\\ 0 & b  \end{pmatrix}, (a,b)\in H,\, n\in\cO_{F,p}\right\}.
\]
$\widetilde{J} :=J/Z_0$ has dimension at least $d_0 +d$, where $d$ comes from the dimension of its unipotent subgroup and $d_0$ comes from the subgroup $\{(a,a^{-1})\in H, a\in\overline{\cO_F^\times}\}$. Since $ \tilde{H}^{j}_{T}$ is an admissible representation of $\mathrm{T}(\Q_p)/H$ we can find a $(\cO_{F,p}^\times)^2$-equivariant map
\[\phi:\tilde{H}^{j}_{T}\to C((\cO_{F,p}^\times)^2/H,\Z_p)^{\oplus l}\]
for some $l$ which is injective after inverting $p$. 
As representations of $\GL_2(\cO_{F,p})$, we have
\begin{eqnarray*}\mathrm{Ind}^{\rmG(\Q_p)}_{\mathrm{B}(\Q_p)} \tilde{H}^{j}_{T} \cong \mathrm{Ind}^{\GL_2(\cO_{F,p})}_{\mathrm{B}(\Q_p)\cap\GL_2(\cO_{F,p})} \tilde{H}^{j}_{T} &\stackrel{\mathrm{Ind}\phi}{\to}& \mathrm{Ind}^{\GL_2(\cO_{F,p})}_{\mathrm{B}(\Q_p)\cap\GL_2(\cO_{F,p})} C((\cO_{F,p}^\times)^2/H,\Z_p)^{\oplus l}\\
&=&C(J\setminus\GL_2(\cO_{F,p}),\Z_p)^{\oplus l}\\
&=&C(\widetilde{J}\setminus\widetilde{\GL_2(\cO_{F,p})},\Z_p)^{\oplus l}.
\end{eqnarray*}
Again $\mathrm{Ind}\phi$ is injective after inverting $p$. Hence it's easy to see from the Gelfand-Kirillov dimension point of view that 
\[j_{\Q_p[[\widetilde{K_p}]]}(\mathrm{Ind}^{\rmG(\Q_p)}_{\mathrm{B}(\Q_p)} (\tilde{H}^{j}_{M,\Q_p})^*)\geq \dim \widetilde{J}\geq d+d_0\geq d+1.\]
\end{proof}

\begin{rem}
When $d\geq 2$, \cite[Theorem 4.9]{HJ23} shows that $\tilde{H}_c^d(K^p)\cong \tilde{H}^d(K^p)$ assuming Leopoldt's conjecture. The authors also mentioned (Remark 4.11 loc.cit.) the possibility of establishing the isomorphism using the known bounds for the Leopoldt's conjecture. In fact the trivial bound will be enough here (like in our proof) and thus one can deduce \eqref{cdv2} from  \eqref{cdv1}. 
\end{rem}

\begin{rem}
We saw that the second isomorphism in Theorem \ref{mtref}
 is very easy to show for modular curves because of the vanishing of $\tilde{H}^0_c$. Emerton observed \cite[Remark 5.4.2]{Eme1} that this simple fact implies the density of algebraic vectors in completed cohomology, which is a crucial ingredient used in his proof of the local-global compatibility. One motivation of our work is to better understand and generalize Emerton's argument.
\end{rem}

\begin{cor}
Under the same assumption, as admissible Banach representations of $\widetilde{K_p}$, both $\tilde{H}_c^d(K^p)_{\Q_p}$ and $\tilde{H}^d(K^p)_{\Q_p}$ can be written as quotients of $C(\widetilde{K_p},\Q_p)^{\oplus m}$ for some $m$.
\end{cor}

This can be viewed as a density result. See Corollary \ref{ccden} below.

\begin{proof}
Write $\Lambda=\Z_p[[\widetilde{K_p}]]$.
By admissibility of $\tilde{H}_c^d(K^p)_{\Q_p}$, its dual $\tilde{H}_c^d(K^p)_{\Q_p}^*$ is a finitely generated module of $\Lambda[\frac{1}{p}]$. Hence we can find a quotient map $\Lambda[\frac{1}{p}]^{\oplus m}\to \tilde{H}_c^d(K^p)_{\Q_p}^*$ of $\Lambda[\frac{1}{p}]$-modules. Taking $E^0_{\Lambda[\frac{1}{p}]}$, we get an inclusion of $\Lambda[\frac{1}{p}]$-modules 
\[E^0_{\Lambda[\frac{1}{p}]}(\tilde{H}_c^d(K^p)_{\Q_p}^*) \subseteq \Lambda[\frac{1}{p}]^{\oplus m}.\] 
By Theorem \ref{mtref}, the first term is isomorphic to $\tilde{H}^d(K^p)_{\Q_p}^*$. Hence taking the continuous $\Hom$ to $\Q_p$, we get a surjective map $C(\widetilde{K_p},\Q_p)^{\oplus m}\to \tilde{H}^d(K^p)_{\Q_p}$. (Here we use $\Lambda[\frac{1}{p}]^*=C(\widetilde{K_p},\Q_p)$). The other claim for $\tilde{H}_c^d(K^p)_{\Q_p}$ can be proved similarly.
\end{proof}

We can also allow a central character. 
Let $E$ be a finite extension of $\Q_p$ and 
\[\chi:Z(\widetilde{K_p})\to E^\times\] 
a continuous character of the center of $\widetilde{K_p}$.
We denote by $\tilde{H}^d(K^p)_{\chi}$ the $\chi$-isotypic subspace of $\tilde{H}^d(K^p)_{E}$ and define $\tilde{H}_c^d(K^p)_{\chi}$ and $C(\widetilde{K_p},\chi)\subseteq C(\widetilde{K_p},E)$ similarly. The dual $E[[\widetilde{K_p}]]_\chi$ of $C(\widetilde{K_p},\chi)$ is a natural algebra quotient of $\Z_p[[\widetilde{K_p}]]\otimes_{\Z_p} E$.

\begin{cor} \label{ccden}
Suppose $K^pK_p$ is neat. There are natural isomorphisms
\[\Hom_{E[[\widetilde{K_p}]]_\chi}(\tilde{H}^d(K^p)_{\chi}^*,E[[\widetilde{K_p}]]_\chi)=\tilde{H}_c^d(K^p)_{\chi}^*,\]
\[\Hom_{E[[\widetilde{K_p}]]_\chi}(\tilde{H}_c^d(K^p)_{\chi}^*,E[[\widetilde{K_p}]]_\chi)=\tilde{H}^d(K^p)_{\chi}^*.\]
Consequently,  as admissible Banach representations of $\widetilde{K_p}$, both $\tilde{H}_c^d(K^p)_{\chi}$ and $\tilde{H}^d(K^p)_{\chi}$ can be written as quotients of $C(\widetilde{K_p},\chi)^{\oplus m}$ for some $m$. If moreover $K_p=\GL_2(\cO_{F,p})$, and $\chi$ induces an algebraic character of $\cO_{F,p}^\times$, the $\GL_2(\cO_{F,p})$-algebraic vectors are dense in both $\tilde{H}_c^d(K^p)_{\chi}$ and $\tilde{H}^d(K^p)_{\chi}$.
\end{cor}

\begin{proof}
We will only prove the first isomorphism  as the same argument works for the other one.
Let $K=K^pK_p$. Fix a connected component $\Sh^\circ_{K,\bC}$ of $\Sh_{K,\bC}$. We can consider its completed cohomology
\[
\varprojlim_n\varinjlim_{K'_p\subseteq \rmG(\Q_p)} H^i(\pi_{K_p'}^{-1}(\Sh^\circ_{K,\bC}),\Z/p^n),
\]
where $\pi_{K_p'}:\Sh_{K^pK_p,\bC}\to \Sh_{K,\bC}$ denotes the projection morphism, and similarly the completed cohomology with compact support. Since $\tilde{H}^*(K^p)$ can be written as the direct sum of the completed cohomology of each component,  it suffices to prove Corollary \ref{ccden} for each connected component. To ease notation, we will assume from now on that $\Sh_{K,\bC}$ is connected.

Choose an element $x$ in $\pi_0(K^p):=\varprojlim_{K_p'} \pi_0(\Sh^\circ_{K^pK'_p,\bC})$ equivalently a compatible system of connected components $\Sh^\circ_{K^pK_p',\bC}$ of $\Sh_{K^pK_p',\bC}$. Consider the completed cohomology of this tower:
\[
\tilde{H}^i(\Sh_K^\circ)=\varprojlim_n\varinjlim_{K'_p\subseteq \rmG(\Q_p)} H^i(\Sh^\circ_{K^pK'_p,\bC},\Z/p^n).
\]
It is a representation of the stabilizer  of $\widetilde{K_p}$ at $x\in\pi_0(K^p)$ which we denote by $K_p^\circ$. Since $\widetilde{K_p}$ acts transitively on $\pi_0(K^p)$ by our assumption on the connectedness of $\Sh_{K,\bC}$, there is an isomorphism of $\widetilde{K_p}$-representations 
\begin{eqnarray} \label{indcc}
\tilde{H}^i(K^p)\cong \Ind_{K_p^\circ}^{\widetilde{K_p}} \tilde{H}^i(\Sh_K^\circ).
\end{eqnarray}
$\SL_2(F_p)$ acts trivially on component groups \cite[\S 2]{De71a}. We get a map
\[K_p\cap  \SL_2(F_p) \to K_p^\circ\]
which turns out to be locally an isomorphism \cite[Lemma 6.2.3]{LP25}. Hence $\Lie K_p^\circ$ is a semisimple Lie algebra. We thus obtain the following lemma.
\begin{lem} \label{Kpcirc}
$K_p^\circ\cap Z(\widetilde{K_p})$ is finite and $K_p^\circ\cdot Z(\widetilde{K_p})$ is an open subgroup of $\widetilde{K_p}$.
\end{lem}

Let $A=E[[\widetilde{K_p}]]$, $A_0=\cO_E[[K_p^\circ]]$, $M=\Hom_{\Z_p}(\tilde{H}^d(\Sh_K^\circ),\cO_E)$ and $\tilde{M}=\tilde{H}^d(K^p)_E^*$. The central character $\chi$ defines an (two-sided) ideal $I$ of $A$. Then $\tilde{H}^d(K^p)_\chi^*=\tilde{M}/I\tilde{M}$. In view of Theorem \ref{mtref}, it is enough to prove 
\[\Hom_A(\tilde{M},A)/I \Hom_A(\tilde{M},A)\cong \Hom_A(\tilde{M},A/I).\]
Now the isomorphism \eqref{indcc} becomes
\[\tilde{M}\cong A\otimes_{A_0} M.\]
Hence $\Hom_A(\tilde{M},A)=\Hom_{A_0}(M,A)$ and we only need to prove the following lemma.
\begin{lem}
For any $A_0$-module $M$,
\[\Hom_{A_0}(M,A)/I \Hom_{A_0}(M,A)\cong \Hom_{A_0}(M,A/I).\]
\end{lem}
\begin{proof}
First we assume $Z(\widetilde{K_p})\cong \Z_p^m$ (equivalently $Z(\widetilde{K_p})$ is torsion-free). Fix topological generators $g_1,\cdots,g_m$.
The ideal $I$ is generated by $x_i:=g_i-\chi(g_i)$, $i=1,\cdots,m$. By Lemma \ref{Kpcirc}, $K_p^\circ\cdot Z(\widetilde{K_p})\cong K_p^\circ\times Z(\widetilde{K_p})$ has finite index, say n, in $\widetilde{K_p}$. We have an isomorphism
\[A\cong A_0[[x_1,\cdots,x_m]]^{\oplus n}\otimes_{\cO_E} E\]
as $\cO_E[[K_p^\circ\times Z(\widetilde{K_p})]]$-modules,
and the lemma follows easily.

In general fix an isomorphism $Z(\widetilde{K_p})\cong \Z_p^m\times \Gamma$ where $\Gamma$ is a finite group. Let $I_0$ denote the ideal of $A$ generated by $g-\chi(g)$, $g\in \Z_p^m\subseteq Z(\widetilde{K_p})$. The argument in the torsion-free case shows that 
\[\Hom_{A_0}(M,A)/I_0 \Hom_{A_0}(M,A)\cong \Hom_{A_0}(M,A/I_0).\]
Now both sides are modules over  $E[\Gamma]$, which is a semisimple algebra. We complete the proof by modulo  the maximal ideal of $E[\Gamma]$ corresponding to $\chi|_{\Gamma}$.
\end{proof}

For the last density claim, it suffices to prove the following.
\end{proof}

\begin{lem}
 $\GL_2(\cO_{F,p})$-algebraic vectors are dense in $C(\GL_2(\cO_{F,p}),\chi)$ if $\chi:\cO_{F,p}^\times \to E^\times$ is an algebraic character.
\end{lem}

\begin{proof}
If $\chi$ is trivial, then we are reduced to case of $\PGL_2$ which was proved in \cite[Lemma A.1.]{Pas14}. The general case can be argued as follows. Let $\mathcal{C}_1$ denote the sheaf of $E$-valued continuous functions on $\PGL_2(\cO_{F,p})$.  The sheaf of $E$-valued continuous functions on $\GL_2(\cO_{F,p})$ with central character $\chi$ descends to a sheaf on $\PGL_2(\cO_{F,p})$ which will be denoted by $\mathcal{C}_\chi$. Clearly $\mathcal{C}_\chi$ is an invertible $\mathcal{C}_1$-module. Let $\mathcal{C}'_\chi \subseteq \mathcal{C}_\chi$ be the closure of algebraic functions in $C(\GL_2(\cO_{F,p}),\chi)$. It suffices to show $\mathcal{C}'_\chi =\mathcal{C}_\chi $. The case of $\PGL_2$ implies that $\mathcal{C}'_\chi $ is a sub $\mathcal{C}_1$-module. Now take a non-zero algebraic function $f$ in $C(\GL_2(\cO_{F,p}),\chi)$. On its non-vanishing locus $U$, $f$ generates  $\mathcal{C}_\chi $ as  a $\mathcal{C}_1$-module, hence $\mathcal{C}_\chi'|_U=\mathcal{C}_\chi|_U$. Using translates of $f$ by $\GL_2(\cO_{F,p})$, we see that $\mathcal{C}_\chi'=\mathcal{C}_\chi$ everywhere.
(This argument essentially appeared in the proof of Proposition 3.2.9 of \cite{Pan22FM}.)
\end{proof}

\bibliographystyle{amsalpha}

\bibliography{bib}

\end{document}